\documentclass[12pt]{article}
\usepackage{graphicx}
\usepackage{color}
\usepackage{enumerate}
\usepackage{latexsym}
\usepackage[T1]{fontenc}
\usepackage{float}
\usepackage{amssymb,amscd,amsmath,amsfonts,amsthm}
\usepackage{multicol}
\usepackage{tabularx,environ,array}
\usepackage{tikz}
\usepackage{comment}
\usepackage{hyperref}
\usepackage{boxedminipage}

\hypersetup{colorlinks=true, linkcolor=blue, citecolor=blue, urlcolor=blue}

\newtheorem{theorem}{Theorem}

\newtheorem{conjecture}{Conjecture}
\newtheorem{corollary}[theorem]{Corollary}

\newtheorem{lemma}[theorem]{Lemma}

\begin{document}

\title{Settling the total domination-annihilation conjecture for graphs with minimum degree two}

\author{Marko Jakovac$^{a,b,}$\thanks{\texttt{marko.jakovac@um.si}}}
\maketitle

\begin{center}
\vspace*{-0.7cm}
$^a$ Faculty of Natural Sciences and Mathematics, University of Maribor, Slovenia\\

\medskip
$^b$ Institute of Mathematics, Physics and Mechanics, Ljubljana, Slovenia
\end{center}

\begin{abstract}
The total domination number $\gamma_t(G)$ of a graph $G$ is the minimum cardinality of a set $D\subseteq V(G)$ such that every vertex of $G$ has a neighbor in $D$. The annihilation number $a(G)$ is the largest integer $k$ for which the sum of the $k$ smallest degrees of $G$ is at most $|E(G)|$. A well-known conjecture, originating from Graffiti.pc and later formulated explicitly by Desormeaux, Haynes, and Henning, asserts that $\gamma_t(G)\le a(G)+1$ for every connected nontrivial graph $G$. The conjecture is known for graphs of minimum degree at least three and for several classes of graphs having vertices of degree one or two. In this paper we settle the minimum-degree-two case. More precisely, we prove $\gamma_t(G)\le a(G)+1$ for every connected graph $G$ with $\delta(G)=2$. The proof combines two sharp bounds on the total domination number with an estimate for the annihilation number. Moreover, in some specific cases, the stronger inequality $\gamma_t(G)\le a(G)$ holds. 
\end{abstract}

\vspace{5mm}


\vspace{5mm}

\noindent
\textbf{Math.\ Subj.\ Class.\ 2020:} 05C69.


\section{Introduction}

All graphs considered in this paper are finite, simple, undirected, and nontrivial (i.e., they have at least two vertices). For a graph $G$, let $V(G)$ and $E(G)$ denote its vertex and edge sets, and let $n(G)=|V(G)|$ and $m(G)=|E(G)|$ denote its order and size, respectively. The degree of a vertex $v$ is denoted by $d_G(v)$, or simply by $d(v)$ when the graph is clear from the context. The minimum and maximum degrees are denoted by $\delta(G)$ and $\Delta(G)$. For a set $X\subseteq V(G)$, we denote by $G[X]$ the subgraph of $G$ induced by $X$.

A set $D\subseteq V(G)$ is a \emph{total dominating set} of $G$ if every vertex of $G$ is adjacent to a vertex of $D$. The minimum cardinality of a total dominating set is the \emph{total domination number} of $G$, denoted by $\gamma_t(G)$. Total domination was introduced by Cockayne, Dawes, and Hedetniemi~\cite{CockayneDawesHedetniemi1980} and has subsequently developed into an extensive area of domination theory; see, for instance,~\cite{HenningYeoBook}.

Let $v_1,v_2,\ldots,v_{n(G)}$ be an ordering of the vertices of $G$ such that $d(v_1)\le d(v_2)\le\cdots\le d(v_{n(G)})$. For simplicity, let $d_i=d_G(v_i)$ for every $i \in \{1,\ldots,n(G)$\}. Thus,
$$d_1\le d_2\le\cdots\le d_{n(G)}$$
is the nondecreasing degree sequence of $G$. The \emph{annihilation number} of $G$, introduced by Pepper \cite{Pepper2004,Pepper2009}, is
$$a(G)=\max\left\{k :\ \sum_{i=1}^{k}d_i\le m(G)\right\}.$$
The following conjecture was posed in a slightly different form by Graffiti.pc and was later reformulated by Desormeaux, Haynes, and Henning.

\begin{conjecture}[\cite{DesormeauxHaynesHenning2013}]\label{conj:main}
If $G$ is a connected nontrivial graph, then
$$\gamma_t(G)\le a(G)+1.$$
\end{conjecture}

It follows directly from the definition that $a(G)\ge \left\lfloor\frac{n(G)}{2}\right\rfloor$. On the other hand, every graph of minimum degree at least three satisfies
$$\gamma_t(G)\le \left\lfloor\frac{n(G)}2\right\rfloor,$$
see, for instance, \cite{ArchdeaconEtAl2004}. Hence, Conjecture \ref{conj:main} holds when $\delta(G)\ge3$. It was proved in \cite{DesormeauxHaynesHenning2013} that the conjecture holds for trees. The result was subsequently extended to cactus graphs and block graphs in \cite{BujtasJakovac2019}, and to a broader class of tree-like graphs in \cite{YueZhuWei2020}. Further results for quasi-trees and several graph constructions were obtained in \cite{HuaHuaKlavzarXu2022}, while additional graph classes with minimum degree at most two were considered in \cite{HuaXuHua2023}. These results indicate that the main difficulty of the conjecture lies in the cases $\delta(G)\in\{1,2\}$.

In this paper, we settle the minimum-degree-two case by proving that $\gamma_t(G)\le a(G)+1$ for every connected nontrivial graph $G$ with $\delta(G)=2$. In fact, under an additional condition on the degree sequence, we obtain the stronger bound $\gamma_t(G)\le a(G)$. The proof uses two known results on total domination. First, we use Henning's bound for connected graphs with minimum degree at least two \cite{Henning2000}. We then apply the Henning--Yeo bound, which takes into account certain path components formed by vertices of degree two \cite{HenningYeo2007}. The main observation of this paper is that the correction term in this bound can be controlled by the annihilation number when the first vertex outside an optimal annihilation set has degree at least three.

\medskip

The paper is organized as follows. In Section~\ref{sec:prelim} we recall the definition of an annihilation set, introduce the cutoff degree, and state the two total-domination bounds needed later. In Section~\ref{sec:main}, we prove the key estimate for the annihilation number and use it to establish the main theorem. We finish in Section~\ref{sec:remarks} with the sharpness of the result and its consequence for the general conjecture.


\section{Preliminaries}\label{sec:prelim}

For a vertex set $X\subseteq V(G)$, we define
$$\Sigma(X,G)=\sum_{v\in X}d_G(v).$$
A set $X$ is an \emph{annihilation set} if $\Sigma(X,G)\le m(G)$. An annihilation set $X$ is \emph{optimal} if $|X|=a(G)$ and
$$\max\{d(v):\ v\in X\}\le \min\{d(u):\ u\in V(G)\setminus X\}.$$
Equivalently, an optimal annihilation set consists of $a(G)$ vertices of smallest possible degrees. For a connected nontrivial graph $G$, we have $a(G)\le n(G)-1$, since the sum of all degrees is $2m(G)>m(G)$. We may therefore define the \emph{cutoff degree}
$$d^*(G)=d_{a(G)+1}.$$
Thus $d^*(G)$ is the smallest degree among vertices outside an optimal annihilation set.

\medskip

We first recall an upper bound due to Henning on the total domination number in terms of the size of a graph.

\begin{theorem}[\cite{Henning2000}]\label{thm:Henning-size}
Let $G$ be a connected graph of size $m(G)$ with $\delta(G)\ge 2$. Then
$$\gamma_t(G)\le \frac{m(G)+2}{2}.$$
\end{theorem}

We next recall an upper bound on the total domination number in terms of the order of the graph and the structure formed by its vertices of degree two. We state it in the notation needed here. Let $G$ be a connected graph with $\delta(G)\ge2$ and $\Delta(G)\ge3$. We call a vertex \emph{small} if it has degree $2$ and \emph{large} if it has degree greater than $2$. Let $S$ and $L$ denote the sets of small and large vertices, respectively. Since $G$ is connected and $L\ne\emptyset$, every component of $G[S]$ is necessarily a path. If some component of $G[S]$ were a cycle, then all its vertices would have both neighbors inside the cycle. Since these vertices have degree two in $G$, the cycle could not be connected to any vertex outside it, which would contradict the connectedness of $G$.

Let $P$ be a component of $G[S]$. Following the definition in \cite{HenningYeo2007}, $P$ is called
\begin{itemize}
    \item a \emph{$0$-path} if $|V(P)|\equiv0\pmod4$ (and hence $|V(P)|\ge 4$) and the two ends of $P$ are adjacent either to the same large vertex or to two distinct adjacent large vertices;
    \item a \emph{$1$-path} if $|V(P)|\ge5$, $|V(P)|\equiv1\pmod4$, and the two ends of $P$ are adjacent to the same large vertex;
    \item a \emph{$3$-path} if $|V(P)|\equiv3\pmod4$.
\end{itemize}
For $i\in\{0,1,3\}$, let $p_i(G)$ be the number of $i$-paths, and let us define
$$p(G)=p_0(G)+p_1(G)+p_3(G).$$
Observe that every path counted by $p(G)$ contains at least three small vertices. Since these paths are distinct components of $G[S]$, we have the inequality
\begin{equation}\label{eq:p-s}
3p(G)\le |S|.
\end{equation}
This inequality will be used in the proof of the main result.

The quantity $p(G)$ appears in the following upper bound on the total domination number due to Henning and Yeo.

\begin{theorem}[\cite{HenningYeo2007}]\label{thm:HY}
Let $G$ be a connected graph of order $n(G)$ with $\delta(G)\ge 2$ and $\Delta(G)\ge 3$. Then
$$\gamma_t(G)\le \frac{n(G)+p(G)}{2}.$$
\end{theorem}


\section{The main result}\label{sec:main}

We first prove a lemma for the case $d^*(G)\ge3$. It gives a lower bound on the annihilation number in terms of $n(G)$ and the quantity $p(G)$ from Theorem~\ref{thm:HY}.

\begin{lemma}\label{lem:d*3}
Let $G$ be a connected graph of order $n$ with $\delta(G)\ge 2$ and $\Delta(G)\ge 3$. If $d^*(G)\ge3$, then
$$a(G)\ge \left\lfloor\frac{n(G)+p(G)}2\right\rfloor.$$
\end{lemma}

\begin{proof}
Let $S$ be the set of degree-two vertices and $L=V(G)\setminus S$ the set of large vertices. Define $s=|S|$, $r=|L|$, and let
$$\ell=\sum_{v\in L}d(v)$$
be the degree sum of the large vertices. Since every vertex in $S$ has degree $2$, we have
\begin{equation}\label{eq:m-l}
2m(G)=2s+\ell, \qquad m(G)=s+\frac{\ell}{2}.
\end{equation}
Moreover,
\begin{equation}\label{eq:l-3r}
 \ell\ge3r,
\end{equation}
because every vertex of $L$ has degree at least $3$.

The assumption $d^*(G)\ge 3$ implies that all $s$ degree-two vertices occur among the first $a(G)$ vertices in the nondecreasing degree sequence. In particular, they are contained in an optimal annihilation set. Therefore $2s\le m(G)$. Using \eqref{eq:m-l}, we get
\begin{equation}\label{eq:l-2s}
\ell \ge 2s.
\end{equation}

We now order the degrees of the large vertices as
$$e_1\le e_2\le\cdots\le e_r, \qquad \sum_{i=1}^{r}e_i=\ell.$$
Set
$$t=\left\lfloor \frac{r}{2}-\frac{r\cdot s}{\ell}\right\rfloor.$$
By \eqref{eq:l-2s}, we have $t\ge 0$. Suppose first that $t\ge 1$. Since $e_1,\ldots,e_t$ are the $t$ smallest degrees, their average is at most the average of all $r$ degrees. Hence,
$$\frac{1}{t}\sum_{i=1}^{t}e_i\le\frac{\ell}{r},$$
and therefore
$$\sum_{i=1}^{t}e_i\le\frac{t\cdot\ell}{r}.$$
If $t=0$, the sum is considered zero and hence, in both cases,
$$\sum_{i=1}^{t}e_i\le\frac{t\cdot\ell}{r}.$$
Consequently,
$$2s+\sum_{i=1}^{t}e_i \le 2s+\frac{t \cdot \ell}{r} \le 2s+\left(\frac{r}{2}-\frac{r \cdot s}{\ell}\right)\frac{\ell}{r}=s+\frac{\ell}{2}=m(G),$$
where the last equality follows from \eqref{eq:m-l}. Hence, the set consisting of all $s$ vertices of degree two together with the $t$ large vertices of smallest degrees is an annihilation set. Therefore,
\begin{equation}\label{eq:a-st}
a(G)\ge s+t =\left\lfloor s+\frac r2-\frac{r \cdot s}{\ell}\right\rfloor.
\end{equation}

It remains to estimate the expression inside the floor. Using $n(G)=s+r$, we obtain
$$s+\frac r2-\frac{r \cdot s}{\ell}=\frac{s+r}{2}+\frac{s}{2}\left(1-\frac{2r}{\ell}\right)=\frac{n(G)}{2}+\frac{s}{2}\left(1-\frac{2r}{\ell}\right).$$
By \eqref{eq:l-3r},
$$1-\frac{2r}{\ell}\ge 1-\frac{2r}{3r}=\frac{1}{3},$$
and therefore
$$s+\frac r2-\frac{r \cdot s}{\ell}\ge \frac {n(G)}{2}+\frac{s}{6}.$$
Finally, \eqref{eq:p-s} gives $\frac{s}{6}\ge \frac{p(G)}{2}$. Thus
$$s+\frac{r}{2}-\frac{r \cdot s}{\ell} \ge \frac{n(G)+p(G)}{2}.$$
Together with \eqref{eq:a-st}, this yields
$$a(G)\ge\left\lfloor\frac{n+p(G)}2\right\rfloor,$$
as required.
\end{proof}

\medskip

We are now ready to prove the main result of this paper. The proof is based on the two bounds for the total domination number stated in Section~\ref{sec:prelim}, together with the estimate from Lemma~\ref{lem:d*3}. According to the value of the cutoff degree, we will use one of these two bounds.

\begin{theorem}\label{thm:main}
Let $G$ be a connected graph with $\delta(G)=2$. Then
$$\gamma_t(G)\le a(G)+1.$$
Moreover, if $d^*(G)\ge3$, then $\gamma_t(G)\le a(G)$.
\end{theorem}

\begin{proof}
Let
$d_1\le d_2\le\cdots\le d_{n(G)}$
be the nondecreasing degree sequence of $G$. Since $\delta(G)=2$, we have $d^*(G)\ge 2$. We distinguish two cases.

\medskip
\noindent
{\bf Case 1}: $d^*(G)=2$.\\
By definition, $d^*(G)=d_{a(G)+1}$, and therefore
$$d_1=d_2=\cdots=d_{a(G)+1}=2.$$
The maximality of $a(G)$ gives
$$\sum_{i=1}^{a(G)+1}d_i>m(G).$$
Hence $2(a+1)>m(G)$ and, since $m(G)$ is an integer,
\begin{equation}\label{eq:m-2a}
m(G)\le 2a(G)+1.
\end{equation}
By Theorem~\ref{thm:Henning-size},
$$\gamma_t(G)\le\frac{m(G)+2}{2}.$$
Since $\gamma_t(G)$ is an integer, \eqref{eq:m-2a} implies
$$\gamma_t(G)\le \left\lfloor\frac{m(G)+2}{2}\right\rfloor \le \left\lfloor\frac{2a(G)+3}{2}\right\rfloor =a(G)+1.$$

\medskip
\noindent
{\bf Case 2}: $d^*(G)\ge3$. \\
If $\Delta(G)=2$, then $G$ is a cycle and $d^*(G)=2$, contrary to the assumption. Hence $\Delta(G)\ge 3$, and Theorem \ref{thm:HY} implies
$$\gamma_t(G)\le \frac{n(G)+p(G)}{2}.$$
Since $\gamma_t(G)$ is an integer, Lemma \ref{lem:d*3} gives
$$\gamma_t(G)\le \left\lfloor\frac{n(G)+p(G)}2\right\rfloor\le a(G).$$

The two cases exhaust all possible values of $d^*(G)$ and therefore complete the proof of the theorem.
\end{proof}

Since the case $\delta(G)\ge 3$ was already known, the new content of Theorem~\ref{thm:main} can be stated as follows.

\begin{corollary}\label{cor:delta2}
If $G$ is a connected graph with $\delta(G) \ge 2$, then
$$\gamma_t(G)\le a(G)+1.$$
\end{corollary}

\medskip

It is now clear that any counterexample to Conjecture~\ref{conj:main}, if one exists, must have minimum degree one. In particular, every counterexample must contain a pendant vertex.


\section{Sharpness and possible counterexamples}\label{sec:remarks}

The additive constant in Theorem~\ref{thm:main} cannot be removed. Indeed, for a cycle $C_n$, $n \ge 3$,
$$a(C_n)=\left\lfloor\frac{n}{2}\right\rfloor,$$
while
$$\gamma_t(C_n)=
\begin{cases}
 \frac{n}{2}+1, & n\equiv2\pmod4,\\[1mm]
 \left\lceil\frac{n}{2}\right\rceil, & \text{otherwise}.
 \end{cases}
$$
Consequently, if $n\equiv2\pmod4$, then $\gamma_t(C_n)=a(C_n)+1$. Thus the bound in Theorem~\ref{thm:main} is sharp for infinitely many connected graphs of minimum degree two. There are also arbitrarily large connected graphs with minimum degree two which are not cycles and for which equality holds. The following construction was given in \cite{BujtasJakovac2019}. For an integer $k\ge 2$, let $G_1,G_2,\ldots,G_k\cong C_6$ be pairwise vertex-disjoint cycles. For every $i\in\{1,\ldots,k-1\}$, identify one vertex of $G_i$ with one vertex of $G_{i+1}$, choosing the vertices in such a way that the resulting graph $G$ satisfies $\Delta(G)\le 4$. Then $G$ is a connected cactus graph with $\delta(G)=2$ and $m(G)=6k$, $\gamma_t(G)=3k+1$ and $a(G)=3k$. Hence, $\gamma_t(G)=a(G)+1$, so this family gives another infinite class of graphs for which the bound in Theorem~\ref{thm:main} is sharp. The proof of Theorem~\ref{thm:main} also isolates the only degree-sequence situation in which the extra $1$ can be needed. Hence, equality in the conjectured bound for a graph with $\delta(G)=2$ can occur only when $d^*(G)=2$. Both extremal families described above satisfy this necessary condition.

The bound is also sharp for graphs with minimum degree one. In \cite{DesormeauxHaynesHenning2013}, the trees satisfying $\gamma_t(T)=a(T)+1$ were characterized. Thus, there are also graphs with pendant vertices for which equality in Conjecture~\ref{conj:main} is attained.

On the other hand, Theorem~\ref{thm:main}, together with the known case $\delta(G)\ge 3$, shows that any counterexample to Conjecture~\ref{conj:main}, if one exists, must have minimum degree one. Hence, the only remaining case consists of graphs containing pendant vertices. The arguments for minimum counterexamples developed in \cite{BujtasJakovac2019} give strong restrictions on vertices of degree one and two, and may therefore be useful in the study of this final case.


\section*{Acknowledgments}

M.\ Jakovac was supported by the Slovenian Research and Innovation Agency (ARIS) under the grants P1-0297, N1-0285, N1-0431.







\end{document}